\documentclass{article}

\usepackage{amsmath, amssymb, amsthm}

\newtheorem{lemma}{Lemma}

\newtheorem{theorem}{Theorem}
\newtheorem{prop}{Proposition}

\newtheorem*{lemma*}{Lemma}
\newtheorem*{thrm*}{Theorem}

\theoremstyle{remark}
\newtheorem*{remark}{Remark}

\newcommand{\cinf}{C^\infty}
\newcommand{\ccinf}{C_c^\infty}
\newcommand{\supp}{\mathop{\mathrm{supp}}}
\newcommand{\WF}{\mathop{\mathrm{WF}}}
\renewcommand{\Re}{\mathop{\mathrm{Re}}}
\newcommand{\R}{\mathbb{R}}

\usepackage{cite}

\title{Multi-wave imaging in attenuating media}
\author{Andrew Homan \\ Purdue University \\ West Lafayette, Indiana, United States}
\date{}

\begin{document}
\maketitle

\begin{abstract}
We consider a mathematical model of thermoacoustic tomography and other multi-wave imaging techniques with variable sound speed and attenuation.
We find that a Neumann series reconstruction algorithm, previously studied under the assumption of zero attenuation, still converges if attenuation is sufficiently small.
With complete boundary data, we show the inverse problem has a unique solution, and modified time reversal provides a stable reconstruction.
We also consider partial boundary data, and in this case study those singularities that can be stably recovered.
\end{abstract}

\section{Introduction}

A multi-wave imaging technique couples a high-resolution imaging modality with a high-contrast imaging modality through a physical mechanism.
In medical imaging, ultrasound waves are often used, as they can provide sub-millimeter resolution in tissue.
The typical examples of such techniques are thermoacoustic tomography (TAT) and photoacoustic tomography (PAT).
The former uses low-frequency microwaves as the contrast modality; the latter uses near-infrared light.
Both are coupled to ultrasound through the photoacoustic effect.
Acoustic pressure is then recorded using transducers placed along the boundary of the region of interest.

To reconstruct an image from the acoustic data, one solves two coupled inverse problems.
First, one reconstructs the ultrasound source inside the region using the acoustic data on the boundary.
Then, the ultrasound source is used to reconstruct the absorption of the contrast modality within the tissue.
We do not consider the second problem here.
In the case of either TAT or PAT, it is called either quantitative thermoacoustic tomography (QTAT) or quantitative photoacoustic tomography (QPAT).

The problem of determining the source that induces a measured acoustic pressure at the boundary of a region has been well-studied under the assumptions of constant speed and zero attenuation.
In this case, the problem is essentially the inversion of a spherical mean operator \cite{finch, kuchment2}.
However, in tissue these assumptions do not hold.
Variations in either acoustic speed \cite{jin1} or attenuation \cite{cox1, deanben1, patch1, burgholzer1} can cause artifacts in reconstruction.
This has spurred the development of reconstruction methods using models of ultrasound propagation that include variable acoustic speed and attenuation.

The method of time reversal provides such a reconstruction \cite{hristova1, hristova2}.
In this method, one fixes a time $T$ (possibly infinite) and solves a mixed boundary problem for the acoustic wave equation with zero Cauchy data at time $T$ and the measured data along the boundary for $t \in (0, T)$.
This exploits the fact that the acoustic wave equation is invariant under time reversal, i.e., $t \mapsto -t$.
To ensure the measured data is compatible with the Cauchy data, a cut-off function is often used.
This technique has also been extended to the thermoelastic equation \cite{scherzer1, ammari1}.

It is also possible to choose Cauchy data that is compatible with the measured data, thereby avoiding the use of a cut-off function.
In \cite{tat1}, they demonstrate that a modified time reversal of this type provides a stable reconstruction in the acoustic wave model.
It has also been extended to the elastic wave equation \cite{tittelfitz1}.
One goal of this work is to extend modified time reversal to a model including variable attuenation.
An advantage of modified time reversal is that under some circumstances it yields an explicit Neumann series for the ultrasound source.
This Neumann series was studied numerically in \cite{tat3} and was shown to be an improvement over time reversal.

There are several commonly-used models of attenuation in biological tissue \cite{kowar1, kowar2, riviere1}.
Generally speaking, the attenuating effect of a medium depends upon the frequency of the wave propagating through it, with higher frequency waves experiencing more attenuation.
A common approximation posits a power law between frequency and attenuation \cite{roitner1, modgil, treeby1}.
Here we take the damped wave equation as a simple model of ultrasound attenuation.
We follow the undamped model of \cite{bal1} in assuming that the ultrasound source is pulse of the form $f(x)\delta^\prime(t)$.
This approximation can be made in the case when the contrast media is an electromagnetic wave, as it is in both TAT and PAT.
Writing the damped wave operator as $\square_a = \partial_t^2 + a\partial_t - c^2\Delta$, we then consider solutions of
\begin{equation}
 \left\{
 \label{Physical-Model}
 \begin{array}{rcl}
  \square_a u & = & f(x)\delta^\prime(t) \text{\ in\ } \R^{n+1}, \\
  u|_{t<0} & = & 0,
 \end{array}
 \right.
\end{equation}
in the sense of distributions.
Here $c$ is a smooth, positive function representing the sound speed.
If $\Omega$ is our region of interest, we assume $c$ is equal to one on $\R^n \setminus \Omega$.
The attenuation coefficient $a$ is assumed to be a smooth, nonnegative function.
If we take $f$ to be in $H^1_0(\Omega)$, then solutions of the Cauchy problem
\begin{equation}
 \left\{
 \label{Damped-Wave-Equation}
 \begin{array}{rcl}
  \square_a u & = & 0 \text{\ in\ } (0, \infty) \times \R^n, \\
  u|_{t=0} & = & f, \\
  \partial_t u|_{t=0} & = & -af.
 \end{array}
 \right.
\end{equation}
are also solutions of \eqref{Physical-Model} when extended by zero to $(-\infty, 0) \times \R^n$.
To show this, take $u$ to be a smooth solution of \eqref{Damped-Wave-Equation} on $\R^{n+1}$.
Then we may consider the distribution $H(t)u(t, x)$, where $H(t)$ is the Heaviside function.
We have that
\[
 \square_a(Hu) = u\delta^\prime + 2(\partial_t u)\delta + au\delta + (\square_a u)H.
\]
The last term is zero, as $\square_a u = 0$.
Let $\phi \in \ccinf(\R^{n+1})$ be a test function.
Then
\begin{align*}
 \langle \square_a(Hu), \phi \rangle &= \int_{\R^n} \left. \left[ - (\partial_t u) \phi - u (\partial_t \phi) + 2(\partial_t u) \phi + au\phi \right]
\right|_{t=0} \, dx, \\
  &= - \int_{\R^n} f\partial_t \phi|_{t=0}\, dx, \\
  &= \langle f\delta^\prime, \phi \rangle,
\end{align*}
which implies that $u$ satisfies \eqref{Physical-Model}.
This justifies the choice of boundary conditions in \eqref{Damped-Wave-Equation}.

Let $\Gamma = \partial\Omega$ be the boundary of region of interest.
If we have access to pressure data on the entire boundary, we can model the measurements with the operator $\Lambda_a f = u|_{(0, T) \times \Gamma}$.
For the partial boundary case, we will take a cut-off function $\chi$ supported on $(0, T) \times \Gamma^\prime$, with $\Gamma^\prime \subset \Gamma$, and model the partial boundary data with $\chi\Lambda_a$.
As $(f, -af)$ is in the energy space $\mathcal{H} = H^1_0(\Omega) \times L^2(\Omega)$, it follows from \cite{llt} that $\Lambda_a f$ is a bounded operator from $H_D(\Omega)$ to $H^1((0, T) \times \Gamma)$.
The inverse problem is to reconstruct $f$ from the measured data $h = \Lambda_a f$.
In the damped case, the modified time reversal of \cite{tat1} suggests that we construct a potential pseudoinverse $A_a h$ from the backward problem with mixed boundary data,
\begin{equation}
 \label{Backward-Problem}
 \left\{
 \begin{array}{rcl}
  \square_a v & = & 0 \text{\ in\ } (0, T) \times \Omega, \\
  v|_{t=T} & = & \phi, \\
  \partial_t v|_{t=T} & = & 0, \\
  v|_{(0, T) \times \Gamma} & = & h.
 \end{array}
 \right.
\end{equation}
Here $\phi$ satisfies $\Delta \phi = 0$ on $\Omega$ and $\phi|_{\Gamma} = h|_{t=T}$.
By the trace theorem, $h|_{t=T} \in H^{1/2}(\Gamma)$, so the problem is well-posed.
We then define $A_a h = v|_{t=0}$.

Our first goal is to determine whether modified time reversal is an effective reconstruction method for the damped wave model.
In particular, we are interested in the error operator $K_a = A_a\Lambda_a - I$.
In the undamped case, under some geometric assumptions, $K_0$ is a strict contraction \cite{tat1}.
This implies that the Neumann series
\[
 f = \sum_{m=0} K_0^m A_0 h
\]
converges.
This series can be used to numerically reconstruct $f$, and in \cite{tat3} it was shown to be an improvement over the usual time reversal algorithm.
In the second section, we show that for small attenuations the above Neumann series is still convergent.
In the third section, we consider the uniqueness and stability of modified time reversal in the presence of arbitrary attenuation, given complete data.
In the fourth section, we show that in some cases partial data is sufficient to recover the singularities of $f$.
In the last section we prove an a priori estimate for the damped wave equation with mixed boundary data, following an earlier estimate \cite{llt}.

\section{Modified time reversal}

We begin with some preliminary remarks on the damped wave equation.
We will assume that the initial data $(f, -af)$ lies in the energy space $\mathcal{H} = H_D(\Omega) \oplus L^2(\Omega, c^{-2}\, dx^2)$. 
Here, $H_D(\Omega)$ is the completion of $C^\infty_0(\Omega)$ under the norm
\[
 ||f||_{H_D}^2 = \int_\Omega |\nabla f|^2\, dx.
\]
By Poincar\'e's inequality, the norm of $H_D(\Omega)$ is equivalent to that of $H^1_0(\Omega)$. 
Equip $\mathcal{H}$ with the inner product
\[
 \langle (a,b), (c,d) \rangle = \int_\Omega \nabla a \cdot \overline{\nabla c} + c^{-2}b\overline{d}\, dx,
\]
which makes $\mathcal{H}$ into a Hilbert space.

Let $U \subset \R^n$.
For an arbitrary function $w$ that is in both $C^1(0, T; L^2(U))$ and $C(0, T; H^1(U))$, define the local energy functional
\[
 E_U(w, t) = \frac{1}{2} \int_U |\nabla w(t, x)|^2 + c^{-2}|\partial_t w(t, x)|^2\, dx.
\]
Both solutions of the wave equation and the damped wave equation are such functions, and it is well-known that energy is conserved in the first case and nonincreasing in the second case.

For use later on, we state an a priori estimate for solutions of the nonhomogeneous damped wave equation with mixed boundary data.
We defer the proof to the last section.
\begin{prop}
\label{Apriori-Theorem}
Let $\Omega$ be a smooth domain, and $u$ be a solution of
\[
 \left\{
  \begin{array}{rcl}
  \square_a u &=& F \text{\ in\ } (0, T) \times \Omega, \\
  u|_{t=0} &=& f, \\
  \partial_t u|_{t=0} &=& g \\
  u|_{(0, T) \times \Gamma} &=& h
  \end{array}
 \right.
\]
subject to the compatibility condition $h|_{t=0} = f|_{\Gamma}$. Then there exists a constant $C > 0$ such that
\[ 
 \sup_{t \in (0, T)} ||(u, \partial_t u)||_{\mathcal{H}} \le C e^{T||a||_\infty} \left\{||F||_{L^1(0, T; L^2)} + ||f||_{H^1} + ||g||_{L^2} + ||h||_{H^1} \right\}, 
\]
provided the boundary data are sufficiently regular.
\end{prop}

Now, take $u$ to be the solution of \eqref{Damped-Wave-Equation} and take $v$ to be the corresponding solution of \eqref{Backward-Problem} with $h = \Lambda_a f$. 
To study the error operator $K_a = A_a\Lambda_a - I$, we study the function $w = u - v$, which solves the IBVP,
\begin{equation}
 \label{Error-Problem}
 \left\{
 \begin{array}{rcl}
  \square_a w & = & 0 \text{\ in\ } [0, T] \times \Omega, \\
  w|_{t=T} & = & u|_{t=T} - \phi, \\
  \partial w|_{t=T} & = & \partial_t u|_{t=T}, \\
  w|_{[0, T] \times \Gamma} & = & 0.  
 \end{array}
 \right.
\end{equation}
Recall $\phi$ is the harmonic extension of $u|_{\{T\} \times \Gamma}$ to $\Omega$, and so the boundary data is compatible. 
Let $W = (w, \partial_t w)$. 
Then \eqref{Error-Problem} reduces to the first-order system,
\[
 (\partial_t + Q_a)W = 0,
\]
where $Q_a$ is the operator
\[
 Q_a = \left[
 \begin{array}{cc}
  0 & -I \\
  -c^2\Delta & a
 \end{array}
 \right].
\]
The following energy estimate for solutions of \eqref{Error-Problem} is used often in what follows.
\begin{lemma}
 \label{Energy-Estimate}
 Let $\omega$ solve
 \[
  \left\{
  \begin{array}{rcl}
   \square_a \omega & = & 0 \text{\ in\ } [0, T] \times \Omega, \\
   \omega|_{t=T} & = & \omega_0, \\
   \partial_t \omega|_{t=T} & = & \omega_1, \\
   \omega|_{[0, T] \times \Gamma} & = & 0.
  \end{array}
  \right.
 \]
 Then for $0 \le t \le T$,
 \[ E_\Omega(\omega, t) \le e^{2(T-t)||a||_\infty} E_\Omega(\omega, T). \]
\end{lemma}

\begin{proof}
Define the operator
 \[
  e^{-tQ_a} = e^{-t(Q_a + ||a||_\infty I)} e^{t||a||_\infty I}.
 \]
By \cite[Theorem X.48]{reed}, if $\Re(Hf, f) \ge 0$ for all $f \in D(H)$, then $e^{-tH}$ is a $C^0$-semigroup of contractions for $t \ge 0$.
This is true of $Q_a + ||a||_\infty$ as an operator on $\mathcal{H}$. 
Therefore $||e^{-tQ_a}|| \le e^{t||a||_\infty}$ for $t \ge 0$.
One can also verify directly that $\omega(t) = e^{-(T-t)Q_a}(\omega_1, \omega_2)$ is the solution of the problem in the lemma.

Hence for $0 \le t \le T$,
\[
 E_\Omega(\omega, t) \le ||e^{-(T-t)}Q_a(\omega_1, \omega_2)||_{\mathcal{H}}^2 \le e^{2T||a||_\infty}E_\Omega(\omega, T),
\]
yielding the desired energy estimate.
\end{proof}

This is a quantitative estimate of the nonincrease of energy in the damped wave equation, in a time-reversed sense.
When solved backward, the damped wave equation gains energy at most exponentially.
One consequence of this estimate is a bound on $||K_a f||_{H_D}$.

\begin{lemma}
 \label{Error-Norm}
 There exists a constant $C > 0$, independent of $a$, such that 
 \[ 
 ||K_a f||_{H_D} \le (1 + C||a||_\infty^2)^{1/2} e^{T||a||_\infty} ||f||_{H_D}.
 \]
\end{lemma}

\begin{proof}
Again let $W = (w, \partial_t w)$. 
By Lemma \ref{Energy-Estimate},
\begin{align*}
 ||K_a f||_{H_D}^2 &\le ||W(0)||_\mathcal{H}^2 \\
                   &\le e^{2T||a||_\infty} ||W(T)||_{\mathcal{H}}^2.
\end{align*}
Using the fact that $\phi$ is harmonic, we see that
\[
 ||W(T)||_{\mathcal{H}}^2 = ||(u|_{t=T} - \phi, \partial_t u|_{t=T})||_{\mathcal{H}}^2 \le E_\Omega(u, T).
\]
Hence
\[
 ||K_a f||_{H_D}^2 \le e^{2T||a||_\infty}(||f||_{H_D}^2 + ||c^{-1}af||_{L^2}^2).
\]
By Poincar\'e's inequality, $||c^{-1}af||_{L^2}^2 \le C||a||_\infty^2||f||_{H_D}^2$, and so we have
\[ 
 ||K_a f||_{H_D} \le (1 + C||a||_\infty^2)^{1/2}e^{T||a||_\infty}||f||_{H_D}.
\]

By itself, this estimate cannot show $K_a$ is a strict contraction on $H_D(\Omega)$, this being sufficient for the convergence of the Neumann series.
\end{proof}

\subsection{Convergence of the Neumann series}

Let $T_0$ be the supremum of the length of geodesics in $(\Omega, c^{-2}\, dx^2)$.
If the metric is trapping, this is infinite.
In \cite{tat1}, they show that if $T_0 < T < \infty$, then $||K_0|| < 1$.
In that case, the Neumann series
\[
 f = \sum_{m=0}^\infty K_0^m A_0 h
\]
converges. 
This series can be used as an effective numerical reconstruction algorithm \cite{tat3}. 
We extend this convergence result by continuity to the damped wave model, given sufficiently small attenuation.
\begin{theorem}
 \label{Convergence}
 If $T_0 < T < \infty$, then there exists $a_0 > 0$ such that for all $a$ such that $||a||_\infty < a_0$, the Neumann series
 \[
  f = \sum_{m=0}^\infty K_a^m A_a h
 \]
 converges.
\end{theorem}

To show this, we use the continuity of the map $a \mapsto K_a$.
\begin{prop}
 Fix $a_0 > 0$. Then there exists $C > 0$ such that for all $a$ with $||a||_\infty \le a_0$,
 \[
  ||K_0 f - K_a f||_{H_D} \le Ca_0(1+a_0^2)^{1/2} e^{Ta_0}||f||_{H_D}
 \]
\end{prop}

\begin{proof}
To prove this, we return to system of coupled PDEs defining $K_0$ and $K_a$, in order to compare their difference.
Let $K_0 f = w_0|_{t=0}$, where $w_0$ solves,
\[
 \left\{
 \begin{array}{rcl}
  \square_0 u_0 = 0 & \text{in} & [0, T] \times \mathbb{R}^n, \\
  u_0|_{t=0} = f, & & \\
  \partial_tu_0|_{t=0} = 0, & &
 \end{array}
 \right.
\]
\[
 \left\{
 \begin{array}{rclll}
  \square_0 w_0 & = & 0 & \text{in} & [0, T] \times \Omega, \\
  w_0|_{t=T} & = & u_0|_{t=T} - \phi_0, & & \\
  \partial_t w_0|_{t=T} & = & \partial_t u_0|_{t=T}, & & \\
  w_0|_{\partial\Omega \times [0, T]} & = & 0. & &
 \end{array}
 \right.
\]
As before, $\phi_0$ is the harmonic extension of $u_0|_{\{T\}\times\partial\Omega}$. 
Recall $K_a f = w|_{t=0}$, where $w = v - u$ solves \eqref{Error-Problem}, $v$ solves \eqref{Backward-Problem}, and $u$ solves \eqref{Damped-Wave-Equation}. 
Define $\bar{u} = u_0 - u$ and $\bar{w} = w_0 - w$. 
Then these functions satisfy the system of PDE
 \[
  \left\{
  \begin{array}{rclll}
   \square_0 \bar{u} & = & a\partial_t u & \text{in} & [0, T] \times \mathbb{R}^n, \\
   \bar{u}|_{t=0} & = & 0, & & \\
   \partial_t\bar{u}|_{t=0} & = & af, & &
  \end{array}
  \right.
 \]
 \[
  \left\{
  \begin{array}{rclll}
   \square_0 \bar{w} & = & a\partial_t w & \text{in} & [0, T] \times \Omega, \\
   \bar{w}|_{t=T} & = & \bar{u}|_{t=T} - (\phi_0 - \phi), & & \\
   \partial_t \bar{w}|_{t=T} & = & \partial_t \bar{u}|_{t=T}, & & \\
   \bar{w}|_{\partial\Omega \times [0, T]} & = & 0. & &
  \end{array}
  \right.
 \]

Our interest is in $E_\Omega(\bar{w}, 0)$, and to determine it we make use of Proposition \ref{Apriori-Theorem}.
Note that $(\phi_0 - \phi)$ is harmonic.
It is also equal to zero on $\Gamma$, and therefore compatible with the Cauchy data.
As before $E_\Omega(\bar{u}, T)$ is a bound on the energy of the Cauchy data.
By Proposition \ref{Apriori-Theorem},
\[
 E_\Omega(\bar{w}, 0) \le C \left\{ ||a\partial_t w||_{L^1(0, T; L^2(\Omega))}^2 + E_\Omega(\bar{u}, T) \right\}.
\]

Concerning the first term, by Jensen's inequality,
\[
 ||a\partial_t w||_{L^1(0, T; L^2(\Omega))}^2 \le ||a||_\infty^2 \int_0^T ||\partial_t w(s, \cdot)||_{L^2(\Omega)}^2\, dx\, ds.
\]
Then by Lemma \ref{Error-Norm}, we have
\[
 ||\partial_t w(s, \cdot)||_{L^2(\Omega)}^2 \le E_\Omega(w, 0) \le (1 + C||a||_\infty^2)e^{2T||a||_\infty}||f||_{H_D}^2.
\]
Hence
\[
 ||a\partial_t w||_{L^1(0, T; L^2(\Omega))} \le a_0(1 + Ca_0^2)^{1/2}e^{Ta_0}||f||_{H_D}
\]

To estimate the second term, we use Duhamel's principle to represent $\bar{u}$ as the integral
\[
 \bar{u}(t, x) = \int_0^t u(s; t, x)\, ds
\]
where $u(s; t, x)$ satisfies
\[
 \left\{
 \begin{array}{rclll}
  \square_0 u(s; t, x) & = & 0 & \text{in} & [s, t] \times \mathbb{R}^n, \\
  u(s; s, x) & = & af, & & \\
  \partial_tu(s; s, x) & = & a\partial_t u(s, x). & &
 \end{array}
 \right.
\]

By nonincrease of energy we have
\[
 E_\Omega(u(s; \cdot), t) \le E_\Omega(u(s; \cdot), s) \le ||a||_\infty^2 \left[ ||f||_{H_D}^2 + E_\Omega(u, s) \right]
\]
and as before $E_\Omega(u, s) \le (1 + C||a||^2_\infty)e^{2T||a||_\infty}||f||_{H_D}^2$.
Hence by Jensen's inequality again,
\[
 E_\Omega(\bar{u}, T) \le \int_0^T E_\Omega(u(s; \cdot), T)\, ds \le C||a||_\infty^2(1 + ||a||^2_\infty)e^{2T||a||_\infty} ||f||_{H_D}^2.
\]

Combining the contributions of both terms and Theorem \ref{Apriori-Theorem}, we have
\[
 ||K_0 f - K_a f||_{H_D}^2 \le C||a||_\infty^2(1 + ||a||_\infty^2)e^{2T||a||_\infty} ||f||_{H_D}^2
\]
which completes the estimate. 
Note that the constant $C$ does not depend on $||a||_\infty$.
\end{proof}

Now we return to the proof of Theorem \ref{Convergence}.

\begin{proof}
By the previous proposition,
\[
 ||K_a f|| \le ||K_0 f - K_a f|| + ||K_0 f|| \le ||K_0 f|| + Ca_0(1 + a_0^2)^{1/2}e^{Ta_0}||f||_{H_D}.
\]
By assumption $||K_0|| < 1$. 
As $a_0 \to 0$, the second term tends to zero, as $C$ does not depend on $a_0$. 
Hence there exists $a_0$ sufficiently small such that $||K_a|| < 1$ uniformly. 
\end{proof}

If $||K_a|| < 1$, then the Neumann series of Theorem \ref{Convergence} also implies stability in the case of complete data.

\section{Complete data}

In this section, we use microlocal methods to prove that under some geometric assumptions and with complete boundary data, modified time reversal is stable.
This follows directly from the previous section, provided the attenuation is sufficiently small.
Here we show stability without this assumption.
For convenience, we adopt the notation of \cite{tat1} for unit speed geodesics.
We identify $T^*\Omega \setminus 0$ with $T\Omega \setminus 0$, via the Euclidean inner product.

{\bf Notation.} Let $(x_0, \xi_0) \in T^*\Omega \setminus 0$. 
There are two null bicharacteristics of the damped wave equation passing over $(x_0, \xi_0)$.
Both project to a unit speed geodesic (with respect to $c^{-2}\, dx^2$) passing through $x_0$.
In such a situation, we write $\gamma_{x_0, \xi_0}$ for the unit speed geodesic with $\gamma_{x_0, \xi_0}(0) = x_0$ and $\dot{\gamma}_{x_0,\xi_0}(0) = \xi_0$.

In this section and the following one, we assume a priori that $\supp f$ is bounded away from $\Gamma$.
This can always be accomplished in practice by slightly enlarging the region of interest, if necessary.

\subsection{Uniqueness}

First, we show uniqueness.
Define $d(x, \Gamma)$ to be the infimum of the lengths of curves (with respect to $c^{-2}\, dx^2$) starting at $x$ and ending at $\Gamma$.
In the undamped case, \cite[Theorem 2]{tat1} yields uniqueness under the assumption that $\Omega$ is strictly convex and that $T > T_1(\Omega)$, where 
\[
 T_1(\Omega) = \sup_{x \in \Omega} d(x, \Gamma).
\]
This result exploited the invariance of the wave equation under time reversal.
To compensate for the lack of this symmetry in the damped wave model, we show uniqueness under the assumption $T > 2 T_1(\Omega)$.

\begin{theorem}
\label{Complete-Uniqueness}
Assume that $\Omega$ is strictly convex and that $T > 2T_1(\Omega)$.
Let $\Lambda_a f = 0$. 
Then $f = 0$.
\end{theorem}

\begin{proof}
Let $u$ be the solution of the forward problem \eqref{Damped-Wave-Equation}.
We apply Tataru's theorem \cite{tataru} (actually, the special case \cite[Theorem 4]{tat1}).
Let $U$ be a neighborhood of $x_0 \in \R^n \setminus \overline{\Omega}$ with $\overline{U} \cap \overline{\Omega} = \emptyset$.
By assumption, $u = 0$ on $\{0\} \times (\R^n \setminus \Omega)$ and $(0, T) \times \Gamma$, and satisfies the damped wave equation on the exterior of $(0, T) \times \Omega$.
Therefore $u = 0$ on $(0, T) \times U$.
By Tataru's theorem, $u = 0$ on the intersection of the backward light cone at $(T, x_0)$ and the forward light cone at $(0, x_0)$.
Repeating this argument for all $x_0$ on the exterior of $\Omega$ sufficiently close to $\Gamma$, we find that $u = 0$ on
\[
 \left\{ (t, x) : d(x, \Gamma) < \frac{T}{2} - \left| \frac{T}{2} - t \right| \right\}
\]
Since $T > 2T_1(\Omega)$, in particular we see that $u = 0$ on a small neighborhood of $\{T/2\} \times \Omega$.
By Proposition \ref{Apriori-Theorem}, $u = 0$ on $(0, T/2) \times \Omega$, and so $f = 0$.
\end{proof}

\subsection{Stability}

To show the reconstruction is stable even when attenuation is large, we turn to microlocal analysis.
When $a = 0$, stability was shown in \cite[Theorem 3]{tat1} under the assumption that $T$ was sufficiently large so that for every $(x, \xi) \in S^*\Omega$, either $\gamma_{x, \xi}$ or $\gamma_{x, -\xi}$ hit $\Gamma$ before time $T$.
Recall from the previous section that $T_0$ is the supremum of the lengths of geodesics in $(\Omega, c^{-2}\, dx^2)$.
Here we present a stability estimate in the special case where $c^{-2}\, dx^2$ is non-trapping, i.e., $T_0 < \infty$.

That this assumption is not necessary for stability follows from the result for partial data in the next section.
However, the proof in this case is simpler.
Notice that in the non-trapping case, if $T > T_0(\Omega)$, then $T > 2T_1(\Omega)$ as well, and so Theorem \ref{Complete-Uniqueness} implies that $\Lambda_a$ is injective.

In the following theorem, let $\chi \in \cinf(\R \times \Gamma)$ be a fixed cutoff function equal to one on $[0, T_0] \times \Gamma$ and zero in a neighborhood of $\{T\} \times \Gamma$.

\begin{theorem}
\label{Complete-Stability}
Assume $T_0 < T < \infty$. Then the following are true:
\begin{enumerate}
\item $A_a\Lambda_a : H_D(\Omega) \to H_D(\Omega)$ is Fredholm.
\item $A_a\chi\Lambda_a = I + R_a$, where $R_a$ is a smoothing operator.
\item There exists a constant $C > 0$ such that
\[ ||f||_{H_D} \le C||\Lambda_a f||_{H^1}. \]
\end{enumerate}
\end{theorem}

\begin{proof}
First, we show $A_a\Lambda_a$ is Fredholm.
In the second section, we considered the error operator $K_a = f - A_a\Lambda_a f$.
It remains to show that $K_a$ is compact.

Let $u$ be the solution of the forward problem \eqref{Damped-Wave-Equation}.
By propagation of singularities, $\WF(u)$ is the union of null bicharacteristics lying over $\WF(f)$.
Let $h = \Lambda_a f$.
As we assume $T > T_0(\Omega)$, $\WF(u)$ is empty over a neighborhood of $\{T\} \times \Omega$.
Therefore the operator $f \mapsto (u|_{t=T} - \phi, \partial_t u|_{t=T})$ maps $H_D(\Omega)$ to $\cinf(\Omega) \times \cinf(\Omega)$, and hence is compact.
Finally, Proposition \ref{Apriori-Theorem} implies that the operator $(u|_{t=T} - \phi, \partial_t u|_{t=T}) \mapsto w|_{t=0}$ is bounded from $\mathcal{H} \to H_D(\Omega)$.
Therefore $K_a$ is compact, and $A_a\Lambda_a$ is Fredholm.

Let $v$ be the solution of the corresponding backward problem \eqref{Backward-Problem} with $h = \chi\Lambda_a f$, and $w = u - v$ as before.
Evidently $v|_{t=T} = 0$.
Therefore $w$ solves \eqref{Error-Problem} with smooth boundary data, compatible to arbitrary order.
If we define $R_a f = w|_{t=0}$, then $R_a$ maps $H_D(\Omega)$ to $\ccinf(\Omega)$, and $A_a\chi\Lambda_a = I + R_a$.

To show the first inequality, the previous equation shows
\[
 ||f||_{H_D} \le ||A_a\chi\Lambda_a f||_{H_D} + ||R_a f||_{L^2}.
\]
By Lemma \ref{Error-Norm}, $A_a$ is bounded from $H^1(\Gamma) \to H_D(\Omega)$, and so we have
\[
 ||f||_{H_D} \le C^\prime(||\Lambda_a f||_{H^1} + ||f||_{L^2}).
\]
To obtain the stability estimate in the theorem, we recall that since $T > 2T_1(\Omega)$, $\Lambda_a$ is injective.
So by \cite[Proposition V.3.1]{taylor1}, there is another constant $C > 0$ such that 
\[ ||f||_{H_D} \le C||\Lambda_a f||_{H^1}. \]
Therefore the reconstruction is stable.

If in addition $\Omega$ is strictly convex (with respect to the Euclidean metric), the results of the next section show that singularities that propagate through the boundary may be stably recovered; the larger $T$ is, the more singularities become visible.
\end{proof}

Concerning the second part of the theorem, in principle one could choose a different modification of the Cauchy data in \eqref{Backward-Problem} that would be compatible with $h$ to at least first order.
In that case, $K_a$ would be smoothing of at least first order, and no cut-off function would be necessary.

\section{Partial data}

In applications it is often only possible to obtain data on a proper subset of the boundary.
Let $\Gamma^\prime \subset \Gamma$ be relatively open.
We model partial measurements by the operator $\chi\Lambda_a$, where $\chi \in \cinf(\R \times \Gamma)$ is a cut-off function positive on $(0, T) \times \Gamma^\prime$ and zero elsewhere.
In this section, we only show that ``visible'' singularities (in the sense of \cite{kuchment1, tat1}) may be stably recovered from $\chi\Lambda_a f$, by proving an estimate of the form
\[
 ||f||_{H_D} \le C(||\Lambda_a f||_{H^1} + ||f||_{L^2}).
\]
If in addition, we know that $\Lambda_a$ is injective, we could prove stability as in the complete data case.
In the undamped case, \cite[Proposition 2]{tat1} shows uniqueness provided that $d(x, \Gamma^\prime) < T$.
However, their result relies on being able to extend solutions of the wave equation to $(-T, 0) \times \Omega$ by time reversal.
It is not clear how to modify their proof to account for attenuation.
Therefore, we constrain ourselves to showing stable recovery of visible singularities.

Let $\mathcal{K}$ be an open set strictly contained in $\Omega$.
By analogy with the complete data case, let $T_2(\mathcal{K}, \Gamma^\prime)$ be the least time necessary for all $(x, \xi) \in S^*\mathcal{K}$ to have the property that at least one of $\gamma_{x, \xi}$ and $\gamma_{x, -\xi}$ reaches $\Gamma^\prime$ before time $T_2(\mathcal{K}, \Gamma^\prime)$.
In other words, we can say that $\mathcal{K}$ is visible from $\Gamma^\prime$.
We will assume that $T > T_2(\mathcal{K}, \Gamma^\prime)$.
Our goal in this section is to show the following:

\begin{theorem}
\label{Partial-Stability}
Assume $\Omega$ is strictly convex (with respect to the Euclidean metric), and that $T_2(\mathcal{K}, \Gamma^\prime) < T < \infty$.
If in addition, $\supp f \subset \mathcal{K}$, then the following are true:

\begin{enumerate}
\item $A_a\chi\Lambda_a : H_D(\mathcal{K}) \to H_D(\Omega)$ is Fredholm.
\item $A_a\chi\Lambda_a : H_D(\Omega) \to H_D(\Omega)$ is an elliptic pseudodifferential operator of order zero.
\item There exists $C > 0$ independent of $a$ such that
\[
 ||f||_{H_D} \le Ce^{T||a||_\infty}(||\Lambda_a f||_{H^1} + ||f||_{L^2}).
\]
\end{enumerate}
\end{theorem}

\begin{proof}
The first part follows directly from the second part of Theorem \ref{Complete-Stability}, though one may need to increase $T$ slightly to ensure that $\chi = 0$ near $\{t = T\}$.

To show the second part, we construct a geometric optics solution $u$ of \eqref{Damped-Wave-Equation}, and then microlocally construct a parametrix of \eqref{Backward-Problem} given $h = \chi\Lambda_a$.
We defer the proof of this to the following two lemmas.

Finally, provided the previous part, by elliptic regularity we have the estimate
\[
 ||f||_{H_D} \le C( ||A_a\chi\Lambda_a f||_{H_D} + ||f||_{L^2} ).
\]
By Proposition \ref{Apriori-Theorem}, $||A_a||_{H^1(\Gamma) \to H_D(\Omega)} \le Ce^{T||a||_\infty}$.
\end{proof}

The principal symbol of $\square_a$ factors as $(\tau + c(x)|\eta|)(-\tau + c(x)|\eta|)$.
Our geometric optics solutions will then depend on the existence of two phase functions solving the corresponding eikonal equations.
To simplify the proof of the next two lemmas, we will assume that these equations are solvable on the interval $[0, T]$, i.e., that no caustics form.
Afterward, we will remove this assumption.

\begin{lemma}
\label{Forward-Parametrix}
$\Lambda_a: H_D(\Omega) \to H^1(\Gamma)$ is the sum of two Fourier integral operators, $\Lambda_a^+$ and $\Lambda_a^-$, with canonical relations
\[
 \{ (y, \eta, t, x, \tau, \xi): t = \tau_{y, \pm\eta}, x = \gamma_{y, \pm\eta}(t), \tau = -|\dot{\gamma}_{y, \pm\eta}(t)|, \xi = -\pi(\dot{\gamma}_{y, \pm\eta}(t)) \},
\]
where $\pi : T_x^*\R^n \to T_x^*\Gamma$ is the tangential projection.
\end{lemma}

\begin{proof}
We will search for a parametrix $u$ for \eqref{Damped-Wave-Equation} of the form
\[
 u(t, y) = (2\pi)^{-n} \sum_{\sigma = \pm} \int e^{i\phi^\sigma(t, y, \eta)} A^\sigma(t, y, \eta) \hat{f}(\eta)\, d\eta,
\]
where $\phi^\pm$ are phase functions and $A^\sigma$ are order zero amplitudes with asymptotic expansion
\[
 A^\sigma(t, y, \eta) \sim \sum_{j \ge 0} A^\sigma_j(t, y, \eta).
\]
where $A^\sigma_j(t, y, \eta)$ is homogeneous of degree $-j$ in $\eta$.
For $u$ to be a parametrix, we must construct $\phi^\sigma$ and $A^\sigma$ so that $\square_a u \in \cinf, u|_{t=0} = f,$ and $\partial_t u|_{t=0} = -af$.
To satisfy the first requirement, we calculate $\square_a u$:
\[
 \square_a u = (2\pi)^{-n} \sum_{\sigma = \pm} \int e^{i\phi^\sigma} \hat{f} [I_2 + I_1 + I_0]\, d\eta,
\]
where
\begin{align*}
I_2 &= -A^\sigma ((\partial_t \phi^\sigma)^2 - c^2|\nabla_y \phi^\sigma|^2), \\
I_1 &= 2i[(\partial_t \phi^\sigma)(\partial_t A^\sigma) - c^2 \nabla_y\phi^\sigma \cdot \nabla_yA^\sigma] + i A^\sigma \square_a \phi^\sigma, \\
I_0 &= \square_a A^\sigma.
\end{align*}

For the residual term to be smooth, the integrand must decay rapidly with respect to $\eta$. 
If we impose the eikonal equation 
\[
 \left\{
 \begin{array}{rcl}
  \mp \partial_t \phi^\pm & = & c|\nabla_y \phi^\pm| \\
  \phi^\pm|_{t=0} & = & y \cdot \eta
 \end{array}
 \right.
\]
and assume for the moment that a solution exists for $t \in [0, T]$, then evidently $I_2 = 0$. 

We recognize the first term of $I_1$ as the transport operator $X^\sigma$ applied to $A^\sigma$, where
\[
 X^\sigma = 2(\partial_t\phi^\sigma) \partial_t - 2c^2\nabla_y \phi^\sigma \cdot \nabla_y
\]

To control the decay of $I_1 + I_1$ in $\eta$, we write $A^\sigma$ as its asymptotic expansion and consider the coefficients separately.
On $A^\sigma_0$, we can impose the transport equation
\[
 X^\sigma A^\sigma_0 + A^\sigma_0 \square_a \phi^\sigma = 0.
\]
The subsequent terms can be taken to satisfy the lower order transport equations
\[
 X^\sigma A^\sigma_j + A^\sigma_j \square_a \phi^\sigma = -\square_a A^\sigma_{j-1}, j \ge 1.
\]

These equations reduce to ordinary differential equations along characteristics -- the characteristics again being unit speed geodesics -- whenever the eikonal equations are solvable, provided we establish Cauchy data at $t = 0$.
To determine the proper Cauchy data to impose, we return to the Cauchy data of \eqref{Damped-Wave-Equation}.

The condition $u|_{t=0} = f$ implies that when $t = 0$,
\[
 f(y) = (2\pi)^{-n} \int e^{iy\cdot\eta} \hat{f}(\eta)[A^+ + A^-]|_{t=0} \, d\eta,
\]
and accordingly at $t = 0$,
\[
 A^+ + A^- = 1.
\]

The condition $\partial_t u|_{t=0} = -af$ implies that when $t = 0$,
\[
 \int e^{iy\cdot\eta}\hat{f}(\eta)(-a(y))\, d\eta = \int e^{iy\cdot\eta}\hat{f}(\eta) \left[ ic|\eta|(-A^+ + A^-) + \partial_t(A^+ + A^-) \right]\, d\eta.
\]
Therefore at $t = 0$,
\[
 ic|\eta|(-A^+ + A^-) + \partial_t(A^+ + A^-) = -a
\]

These two equations yield the following system of boundary conditions at $t = 0$:
\[
 \left\{
 \begin{array}{rcl}
 A^+_0 + A^-_0 & = & 1, \\
 A^+_1 + A^-_1 & = & 0, \\
 A^+_j + A^-_j & = & 0,\ j \ge 2
 \end{array}
 \right.
\ 
 \left\{
 \begin{array}{rcl}
 A^+_0 - A^+_0 & = & 0, \\
 A^+_1 - A^-_1 & = & -a - \partial_t(A^+_0 + A^-_0), \\
 A^+_j - A^+_j & = & - \partial_t(A^+_{j-1} + A^+_{j-1}),\ j\ge 2.
 \end{array}
 \right.
\]
Note that the resulting system of equations may be solved recursively, yielding boundary data for each coefficient of the asymptotic expansion of $A^\pm$.
In particular, $A^+_0(0, y, \eta) = A^-_0(0, y, \eta) = \frac{1}{2}$. 
With this set of boundary data, the recurring system of transport equations stated above may be solved.
This completes the construction of $u$.

The restriction of this parametrix to $(0, T) \times \Gamma$ yields a global representation of $\Lambda_a$ as the sum of two oscillating integrals, up to a smooth residual error.
For $(t, x) \in (0, T) \times \Gamma$, we have
\[
 [\Lambda^\pm_a f](t, x) = (2\pi)^{-n} \int e^{i\phi^\pm(t, x, \eta)} A^\pm(t, x, \eta) \hat{f}(\eta)\, d\eta.
\]
From this we can calculate the canonical relation of each Fourier integral operator.
The characteristic manifold is 
\[ \Sigma^\pm = \{ y = \partial_\eta \phi^\pm(t, x, \eta) \}. \]
It is well-known that $y = \partial_\eta \phi^\pm(t, x, \eta)$ holds when $\gamma_{y, \pm\eta}(t) = x$. 
Therefore, the canonical relation of $\Lambda^\pm_a$, restricted to $T^*\Omega \times T^*((0, T) \times \Gamma)$, is given locally by the graph of the diffeomorphism mapping $(y_0, \eta_0)$ to $(\tau_{y_0, \pm\eta_0}, \gamma_{y_0, \pm\eta_0}(\tau_{y_0, \pm\eta_0}))$.
\end{proof}

Next we perform a similar construction for the backward problem \eqref{Backward-Problem}, in order to analyze $A_a\chi\Lambda_a$.

\begin{lemma}
\label{Backward-Parametrix}
$A_a\chi\Lambda_a : H_D(\Omega) \to H_D(\Omega)$ is a pseudodifferential operator of order zero whose principal symbol is
\[
 \sigma_0(y, \eta) = \frac{1}{2}[\chi(\tau_{y, \eta}, \gamma_{y, \eta}(\tau_{y, \eta})) + \chi(\tau_{y, -\eta}, \gamma_{y, -\eta}(\tau_{y, -\eta}))]
\]
\end{lemma}

\begin{proof}
As in the previous lemma, we begin by constructing a parametrix $v$ of \eqref{Backward-Problem}, assuming again that the eikonal equations are solvable on $(0, T)$.
To proceed, we microlocalize the parametrix near the point $(t_0, x_0, \tau_0, \xi_0) \in T^*((0, T) \times \Gamma)$.
We take $\rho \in \ccinf(\R^{n+1})$ to be a cutoff function supported in a small neighborhood $U$ of $(t_0, x_0)$ such that $\rho(t_0, x_0) = 1$.
If $\tau_0 < 0$, we will take $h = \rho\chi\Lambda_a^+ f$ and assume that $\WF(h)$ is contained in a small conic neighborhood of the point.
If $\tau_0 > 0$, replace $\Lambda_a^+$ with $\Lambda_a^-$ and proceed similarly
As Cauchy data for \eqref{Backward-Problem}, we will take $v|_{t=T} = \partial_t v|_{t=T} = 0$.
The boundary data is smoothly compatible, as we assumed that $\chi$ is supported away from $\{t = T\}$.

Under these considerations, we search for $v$ solving \eqref{Backward-Problem} up to smooth error with the above initial data.
We assume $v$ is of the form
\[
 v(t, y) = (2\pi)^{-n} \int e^{i\psi(t, y, \eta)} B(t, y, \eta) \hat{f}(\eta)\, d\eta,
\]
as suggested by the forward problem.

Propagation of singularities implies that the wavefront set of $v$ should be near the two null bicharacteristics whose projection onto $T_{(t_0, x_0)}^*((0, T) \times \Gamma)$ is $(\tau_0, \xi_0)$.
The assumption of strict convexity implies that there are two distinct geodesics whose tangent at $x_0$ projects to $\xi_0$; one pointing inward, and one pointing outward.

Let $(y_0, \eta_0)$ be related to $(t_0, x_0, \tau_0, \xi_0)$ under the canonical relation of $\Lambda^+_a$. 
Therefore, the unit speed geodesic $\gamma_{y_0, \eta_0}$ is such that $x_0 = \gamma_{y_0, \eta_0}(\tau_{y_0, \eta_0})$.
Evidently, this is the geodesic pointing outward.

The other geodesic is the reflection of this one along $\Gamma$, according to the laws of geometric optics. 
Singularities near it propagate further, perhaps reflecting off the boundary at most finitely many more times (as the boundary is strictly convex), until they intersect $\{t =T\}$.
Since we specified that the Cauchy data is zero there, these reflected broken geodesics carry no singularities.

Since only the null bicharacteristic corresponding to $\gamma_{y_0, \eta_0}$ carries the singularities of $h$, we need only construct $\psi$ and $B$ in a small conic neighborhood of this bicharacteristic.
The boundary data is of the form
\[
 v|_U = \rho \chi \Lambda_a^+ f = (2\pi)^{-n} \int e^{i\phi^+} \rho \chi A^+ \hat{f}\, d\eta.
\]
For $v$ to be a parametrix, we must have
\[
 \left\{
 \begin{array}{rcl}
 |\partial_t \psi|^2 & = & c^2(x)|\nabla_x \psi|^2, \\
 \psi|_U & = & \phi^+|_U.
 \end{array}
 \right.
\]
This holds in particular if we impose $-\partial_t v = c(x)|\nabla_x \psi|$.
Then $\psi$ and $\phi^+$ are given by the method of characteristics in a small conic neighborhood of the bicharacteristic.
They have the same characteristics and the same boundary data, so in this neighborhood they are equal.

Similarly, $B$ solves the same recurring system of transport equations as $A^+$, though with different initial data; on $U$, $B = \rho\chi A^+$.
Since the first transport equation is homogeneous, we have,
\[ B_0(0, y_0, \eta_0) = \rho(t_0, x_0)\chi(t_0, x_0)A^+_0(0, y_0, \eta_0) = \frac{1}{2}\chi(\tau_{y_0, \eta_0}, \gamma_{y_0, \eta_0}(\tau_{y_0, \eta_0})). \]

We repeat the construction for points in $T^*((0, T) \times \Gamma^\prime)$ with $\tau_0 < 0$, as mentioned before.
Passing to a microlocal partition of unity, the above constructs $v$ as a parametrix for \eqref{Backward-Problem} with boundary data $h = \chi\Lambda_a f$ and zero Cauchy data at $\{t = T\}$.
Restricting $v$ to $\{t = 0\}$ yields a representation of $A_a\chi\Lambda_a$ as a pseudodifferential operator, whose principal symbol is $B_0$.
After accounting for both contributions to $B_0(0, y_0, \eta_0)$ from singularities propagating along $\gamma_{y_0, \eta_0}$ and $\gamma_{y_0, -\eta_0}$, we see that the principal symbol is
\[ B_0(0, y_0, \eta_0) = \frac{1}{2}[\chi(\tau_{y_0, \eta_0}, \gamma_{y_0, \eta_0}(\tau_{y_0, \eta_0})) + \chi(\tau_{y_0, -\eta_0}, \gamma_{y_0, -\eta_0}(\tau_{y_0, -\eta_0}))]. \]
Therefore $A_a\chi\Lambda_a$ is of order zero.
\end{proof}

\begin{remark}
The previous two lemmas assumed that the eikonal equation was solvable on $(0, T) \times \Omega$.
In general, solutions with Cauchy data on $\{t = 0\}$ exist only on a small interval $[0, t_1]$.
To continue past $\{t = t_1\}$, we use the previous parametrix to obtain Cauchy data on this hyperplane.
Then, we solve the eikonal equations with boundary data $\phi^\pm_1|_{t = t_1} = y \cdot \eta$, and the transport equations with Cauchy data given by the previous parametrix.
This will be solvable on some interval $(t_1, t_2]$.
By compactness, there exists a positive lower bound on the length of the interval on which the eikonal equation is solvable.
Therefore, after repeating the construction at most finitely many times (when $T < \infty$), we obtain a parametrix on $(0, T) \times \Omega$.
The backward parametrix is then also constructed in a similar way on the same intervals.
\end{remark}

We now conclude the proof of Theorem \ref{Partial-Stability}. According to Lemma \ref{Backward-Parametrix}, $A_a\chi\Lambda_a$ is a pseudodifferential operator. 
Its symbol is of order zero by construction.
If $\mathcal{K}$ is visible from $\Gamma^\prime$, then the principal symbol comprises two terms; both are nonnegative, and at least one is positive.
Hence $A_a\chi\Lambda_a$ is elliptic.
If $\mathcal{K}$ is not visible, then $A_a\chi\Lambda_a$ is still a pseudodifferential operator, but it is only elliptic at those points of $T^*\Omega \setminus 0$ which are visible from the boundary.

Finally, we note that even if $\mathcal{K}$ is not visible from $\Gamma^\prime$, $A_a\chi\Lambda_a$ remains a pseudodifferential operator, and is still Fredholm by the considerations of Theorem \ref{Complete-Stability}.

\section{A priori estimate}

We return now to the proof of Proposition \ref{Apriori-Theorem}. 
We follow \cite[Theorem 4.2]{llt}, using the method of multipliers to develop an a priori estimate.
There, they studied operators of the form $\partial_t^2 - A$, with $A$ uniformly elliptic and in divergence form. 
For our purposes here we modify their proof to account for the lower order term in the damped wave equation.
Recall $u$ is a solution of
\[
 \left\{
  \begin{array}{rcl}
  \square_a u &=& F \text{\ in\ } (0, T) \times \Omega, \\
  u|_{t=0} &=& f, \\
  \partial_t u|_{t=0} &=& g \\
  u|_{(0, T) \times \Gamma} &=& h
  \end{array}
 \right.
\]
subject to the compatibility condition $h|_{t=0} = f|_{\Gamma}$.

\begin{proof}
The proof has three steps. 
First, we establish an a priori estimate of the form above, except that it also depends on the $L^1(0, T; L^2(\Gamma))$-norm of the normal derivative of $u$, along with $||u||_{H^1}$ and $||\partial_t u||_{L^2}$. 
Next, we remove the first dependency by estimating the normal derivative. 
At this point the estimate contains a term depending on the $L^1(0, T; L^2(\Gamma))$-norm of $\nabla u$, which we separate into its normal and tangential components.
The final dependencies are removed via Gronwall's inequality.

Let $n$ be the outward normal vector field of $\Gamma$.
As $\Gamma$ is smooth, we can extend $n$ to a smooth vector field in a neighborhood of $\Gamma$.
Using a cut-off function, we can then extend it to a smooth vector field on a neighborhood of $\Omega$, which we will also write as $n$. 
The result does not depend on the particular extension chosen.

\emph{Step 1}: We begin by fixing $s \in (0, T)$ in the PDE and multiplying through by $2c^{-2}\partial_t u$, integrating with respect to $\Omega$.
\begin{equation*}
\int_\Omega 2c^{-2} (\partial_t^2 u) \partial_t u - 2 (\Delta u) u_t + 2ac^{-2}(\partial_t u)^2\, dx = \int_\Omega 2c^{-2}(\partial_t u)^2 F\, dx.
\end{equation*}
We apply the identity
\[
 2c^{-2} (\partial_t^2 u) \partial_t u - 2(\Delta u) u_t = \frac{d}{dt}\left( ||c^{-1} \partial_t u||_{L^2}^2 + ||\nabla u||_{L^2}^2 \right) - 2\nabla \cdot (\partial_t u)(\nabla u)
\]
to the first two terms. 
Rearranging, we have
\[ 
 \frac{d}{dt}\left( ||c^{-1} \partial_t u||_{L^2}^2 + ||\nabla u||_{L^2}^2 \right) = \int_\Omega 2\nabla \cdot (\partial_t u)(\nabla u) - 2ac^{-2}(\partial_t u)^2 + 2c^{-2}(\partial_t u) F\, dx.
\]
The second term can be simplified with the divergence theorem. 
After integrating from $0$ to $s$ and freely applying Young's inequality, we have
\begin{align*}
||c^{-1} \partial_t u(s)||_{L^2}^2 + ||\nabla u(s)||_{L^2}^2 &\le C \left\{||F||_{L^1(0, T; L^2)}^2 + ||f||_{H^1}^2 + ||g||_{L^2}^2 + ||h||_{H^1}^2 \phantom{\int} \right. \\ 
  & + \left. \int_0^s \int_{\Gamma} |\nabla u \cdot n|^2\, dS\, dt \right\} \\ 
  & + (1 + ||a||_\infty) \int_0^s ||c^{-1}(\partial_t u)||_{L^2(\Omega)}^2 \,
dt.
\end{align*}
This agrees with the analogous estimate \cite[4.12]{llt}. 
Here $C$ is a generic constant that is independent of $a$.

\emph{Step 2}: To establish the needed estimate on the normal derivative, we repeat the process above, multiplying through by $c^{-2}(\nabla u \cdot n)$ and integrating over $(0, s) \times \Omega$.
\begin{equation}
\label{Normal-Derivative-Equation}
\int_0^s \int_\Omega [c^{-2}\partial_t^2u - \Delta u + c^{-2}a\partial_t u](\nabla u \cdot n) \, dx\, dt = \int_0^s \int_\Omega c^{-2}F(\nabla u \cdot n)\, dx\, dt
\end{equation}
The first two terms reduce nearly the same as the analogous estimates \cite[4.13]{llt} and \cite[4.16]{llt}:
\begin{align*}
\int_0^s \int_\Omega c^{-2}(\partial_t^2u)(\nabla u \cdot n)\, dx\, dt &= \int_\Omega c^{-2}(\partial_t u(s, \cdot))(\partial u(s, \cdot) \cdot n) - g(\nabla f \cdot n) \, dx \\
 &+\int_0^s\int_\Omega \frac{1}{2} (\nabla \cdot (c^{-2}n))(\partial_t u)^2 \,dx\, dt \\
 &-\int_0^s\int_\Omega \partial_t(c^{-2})(\partial_t u)(\nabla u \cdot n)  \, dx\, dt \\
 &-\int_0^s\int_{\Gamma} \frac{1}{2}(\nabla(c^{-2}) \cdot n)(\partial_t h)^2\, dS\, dt,
\end{align*}
\begin{align*}
\int_0^s (\Delta u)(\nabla u \cdot n)\, dx\, dt &= \int_0^s\int_{\Gamma} \frac{1}{2}|\nabla u|^2 - |\nabla u \cdot n|^2\, dS\, dt \\
  &+ \int_0^s\int_\Omega \big\langle (\nabla n)(\nabla u), \nabla u \rangle\, dx\,dt.
\end{align*}
Here we use the fact that 
\[ 
\frac{1}{2}|\nabla u|^2 - |\nabla u \cdot n|^2 = \frac{1}{2}|\nabla u \cdot n^\perp|^2 - \frac{1}{2}|\nabla u \cdot n|^2 \le \frac{1}{2}|\nabla u \cdot n|^2
\]
Isolating this term to the left-hand side, we obtain,
\begin{align*}
\frac{1}{2}\int_0^s\int_{\Gamma} |\nabla u \cdot n|^2 \, dS\, dt &\le \int_\Omega c^{-2}(\partial_t u(s, x))(\nabla u(s, x) \cdot n) - g(\nabla f \cdot n) \, dx \\
 &+\int_0^s\int_\Omega \frac{1}{2} (\nabla \cdot c^{-2}n)(\partial_t u)^2 - (\partial_tc^{-2})(\partial_t u)(\nabla u \cdot n)  \, dx\, dt \\
 &-\int_0^s\int_{\Gamma} \frac{1}{2}(\nabla(c^{-2}) \cdot n)(\partial_t h)^2\, dS\, dt \\ 
 &+\int_0^s\int_\Omega \big\langle (\nabla n)(\nabla u), \nabla u \rangle\, dx\,dt \\ 
 &+\int_0^s\int_\Omega c^{-2}a(\partial_t u)(\nabla u \cdot n)\, dx\, dt \\
 &-\int_0^s\int_\Omega c^{-2}F(\nabla u \cdot n)\, dx\, dt.
\end{align*}
On the right hand side, we majorize with the absolute value and apply Young's inequality to several terms. 
This yields the desired estimate for the normal derivative in terms of the boundary data and the $L^1(0, T; L^2(\Omega))$-norms of $c^{-1}(\partial_t u)$ and $|\nabla u|$. 
Notice it is only these last two terms that depend on $||a||_\infty$, and that the dependence is at most linear.

\emph{Step 3}: We begin from the inequality derived from applying the estimate for the normal derivative to the estimate from step 1.
\begin{align*}
||(u, \partial_t u)(s)||_{\mathcal{H}}^2 &\le C\left\{ ||F||_{L^1(0, T; L^2)}^2 + ||f||_{H^1}^2 + ||g||_{L^2}^2 + ||h||_{H^1}^2 \right\}  \\
  &+ (1 + ||a||_\infty)\int_0^s ||(u, \partial_t u)(t)||_{\mathcal{H}}^2 \, dt 
\end{align*}
By Gronwall's inequality, for all $0 \le s \le T$ we have
\begin{equation*}
||(u, \partial_t u)(s)||_{\mathcal{H}}^2 \le C e^{2s||a||_\infty} \left\{ ||F||_{L^1(0, T; L^2)}^2 + ||f||_{H^1}^2 + ||g||_{L^2}^2 + ||h||_{H^1}^2 \right\}.
\end{equation*}
Finally, we take the supremum over $s \in [0, T]$ to obtain the estimate of Proposition \ref{Apriori-Theorem}.
While we assumed that $\Omega, c,$ and $a$ were all smooth, the proof shows that their being $C^1$ suffices. 
\end{proof}

\bibliographystyle{abbrv} 
\bibliography{multi}
\end{document}